\documentclass[12pt,twoside]{amsart}
\usepackage{amsmath, amsthm, amscd, amsfonts, amssymb, graphicx}
\usepackage[bookmarksnumbered, plainpages]{hyperref}
\usepackage[latin1]{inputenc}

\newtheorem{thm}{Theorem}[section]

\newtheorem{lem}[thm]{Lemma}
\newtheorem{prop}[thm]{Proposition}
\newtheorem{defn}[thm]{Definition}

\numberwithin{equation}{section}

\begin{document}

\title{\bf  Several new two-variable elliptic genera }
\author{Yong Wang}

\thanks{{\scriptsize
\hskip -0.4 true cm \textit{2010 Mathematics Subject Classification:}
58C20; 57R20; 53C80.
 }}

\maketitle

\begin{abstract}
In this paper, we define several new two-variable elliptic genera for even dimensional spin manifolds which are the index for some Dirac operators and holomorphic $SL(2,Z)$, $\Gamma_0(2)$, $\Gamma^0(2)$, $\Gamma _{\theta}$-Jacobi forms. By these Jacobi forms, we can get some $SL(2,{\bf Z})$ and $\Gamma^0(2)$ modular forms. By these $SL(2,{\bf Z})$ and $\Gamma^0(2)$ modular forms, we get some interesting
anomaly cancellation formulas for spin manifolds.
\end{abstract}

\vskip 0.2 true cm


\pagestyle{myheadings}
\markboth{\rightline {\scriptsize Yong Wang}}
         {\leftline{\scriptsize  Several new two-variable elliptic genera}}

\bigskip
\bigskip


\section{ Introduction}
For an arbitrary compact spin manifold one can define its elliptic genus. It is a modular form in one variable with respect to a congruence
subgroup of level $2$. For a compact complex manifold one can define its elliptic genus as a function in two complex variables. In the last case,
the elliptic genus is the holomorphic Euler characteristic of a formal power series with vector bundle coefficients. If the first Chern class of the complex manifold is equal to zero, then the elliptic genus is a weak Jacobi form. In \cite{LiP}, Li extended the elliptic genus of an almost complex
manifold to a twisted version where an extra complex vector bundle is involved. Under some conditions, Li proved this elliptic genus is a weak Jacobi
form. In \cite{Wa1}, we extended the Li's elliptic genus and proved this generalized elliptic genus is not a weak Jacobi form and called it the generalized Jacobi form. In \cite{Wa2}, we extended the Li's elliptic genus to the family case and obtained some new anomaly cancellation formulas for the determinant line bundle and index gerbes, and certain results about eta invariants. In \cite{Wa3}, we defined a two-variable elliptic genus for odd dimensional spin manifolds which is the index for some Toeplitz operator and a holomorphic $SL(2,Z)$-Jacobi form. As corollaries, we got some interesting anomaly cancellation formulas for odd spin manifolds.\\
\indent In 1983, the physicists Alvarez-Gaum\'{e} and Witten \cite{AW}
  discovered the ``miraculous cancellation" formula for gravitational
  anomaly which reveals a beautiful relation between the top
  components of the Hirzebruch $\widehat{L}$-form and
  $\widehat{A}$-form of a $12$-dimensional smooth Riemannian
  manifold. Kefeng Liu \cite{Li1} established higher dimensional "miraculous cancellation"
  formulas for $(8k+4)$-dimensional Riemannian manifolds by
  developing modular invariance properties of characteristic forms.
  These formulas could be used to deduce some divisibility results. In
  \cite{HZ1}, \cite{HZ2}, \cite{CH}, some more general cancellation formulas that involve a
  complex line bundle and their applications were established. In \cite{HLZ2}, Han, Liu and Zhang showed that both of the Green-Schwarz anomaly factorization formula
for the gauge group $E_8\times E_8$ and the Horava-Witten anomaly factorization formula for the gauge
group $E_8$ could be derived through modular forms of weight $14$. This answered a question of J.
H. Schwarz.  In \cite{HHLZ}, Han, Huang, Liu and Zhang introduced a modular form of weight $14$ over $SL(2,{\bf Z})$ and a modular form of weight $10$ over $SL(2,{\bf Z})$ and they got some interesting
anomaly cancellation formulas on $12$-dimensional manifolds.
In this paper, we define several new two-variable elliptic genera for even dimensional spin manifolds which are the index for some Dirac operators and holomorphic $SL(2,Z)$, $\Gamma_0(2)$, $\Gamma^0(2)$, $\Gamma _{\theta}$-Jacobi forms. By these Jacobi forms, we can get some $SL(2,{\bf Z})$ and $\Gamma^0(2)$ modular forms. By these $SL(2,{\bf Z})$ and $\Gamma^0(2)$ modular forms, we get some interesting
anomaly cancellation formulas for spin manifolds.\\
  \indent This paper is organized as follows. In Section 2, we have introduced some definitions and basic concepts that we will use in the paper. In Section 3,  we define two new two-variable elliptic genera for even dimensional spin manifolds which are the index for some Dirac operators and holomorphic $SL(2,Z)$-Jacobi forms and some applications are given. In Section 4, we define some two-variable elliptic genera for even dimensional spin manifolds which are holomorphic $\Gamma_0(2)$, $\Gamma^0(2)$, $\Gamma _{\theta}$-Jacobi forms and some applications are given. \\

 \vskip 1 true cm

 \section{Characteristic Forms and Modular Forms}
\quad The purpose of this section is to review the necessary knowledge on
characteristic forms and modular forms that we are going to use.\\

 \noindent {\bf  2.1 characteristic forms }\\
 \indent Let $M$ be a Riemannian manifold.
 Let $\nabla^{ TM}$ be the associated Levi-Civita connection on $TM$
 and $R^{TM}=(\nabla^{TM})^2$ be the curvature of $\nabla^{ TM}$.
 Let $\widehat{A}(TM,\nabla^{ TM})$ defined by (cf. \cite{Zh})
\begin{equation}
   \widehat{A}(TM,\nabla^{ TM})={\rm
det}^{\frac{1}{2}}\left(\frac{\frac{\sqrt{-1}}{4\pi}R^{TM}}{{\rm
sinh}(\frac{\sqrt{-1}}{4\pi}R^{TM})}\right).
\end{equation}
 Let $E$, $F$ be two Hermitian vector bundles over $M$ carrying
   Hermitian connection $\nabla^E,\nabla^F$ respectively. Let
   $R^E=(\nabla^E)^2$ (resp. $R^F=(\nabla^F)^2$) be the curvature of
   $\nabla^E$ (resp. $\nabla^F$). If we set the formal difference
   $G=E-F$, then $G$ carries an induced Hermitian connection
   $\nabla^G$ in an obvious sense. We define the associated Chern
   character form as
   \begin{equation}
       {\rm ch}(G,\nabla^G)={\rm tr}\left[{\rm
   exp}(\frac{\sqrt{-1}}{2\pi}R^E)\right]-{\rm tr}\left[{\rm
   exp}(\frac{\sqrt{-1}}{2\pi}R^F)\right].
   \end{equation}
   For any complex number $t$, let
   $$\wedge_t(E)={\bf C}|_M+tE+t^2\wedge^2(E)+\cdots,~S_t(E)={\bf
   C}|_M+tE+t^2S^2(E)+\cdots$$
   denote respectively the total exterior and symmetric powers of
   $E$, which live in $K(M)[[t]].$ The following relations between
   these operations hold,
   \begin{equation}
       S_t(E)=\frac{1}{\wedge_{-t}(E)},~\wedge_t(E-F)=\frac{\wedge_t(E)}{\wedge_t(F)}.
   \end{equation}
   Moreover, if $\{\omega_i\},\{\omega_j'\}$ are formal Chern roots
   for Hermitian vector bundles $E,F$ respectively, then
   \begin{equation}
       {\rm ch}(\wedge_t(E))=\prod_i(1+e^{\omega_i}t)
   \end{equation}
   Then we have the following formulas for Chern character forms,
   \begin{equation}
       {\rm ch}(S_t(E))=\frac{1}{\prod_i(1-e^{\omega_i}t)},~
{\rm ch}(\wedge_t(E-F))=\frac{\prod_i(1+e^{\omega_i}t)}{\prod_j(1+e^{\omega_j'}t)}.
   \end{equation}
\indent If $W$ is a real Euclidean vector bundle over $M$ carrying a
Euclidean connection $\nabla^W$, then its complexification $W_{\bf
C}=W\otimes {\bf C}$ is a complex vector bundle over $M$ carrying a
canonical induced Hermitian metric from that of $W$, as well as a
Hermitian connection $\nabla^{W_{\bf C}}$ induced from $\nabla^W$.
If $E$ is a vector bundle (complex or real) over $M$, set
$\widetilde{E}=E-{\rm dim}E$ in $K(M)$ or $KO(M)$.\\

\noindent{\bf 2.2 Some properties about the Jacobi theta functions
and modular forms}\\
   \indent We first recall the four Jacobi theta functions are
   defined as follows( cf. \cite{Ch}):
   \begin{equation}
      \theta(v,\tau)=2q^{\frac{1}{8}}{\rm sin}(\pi
   v)\prod_{j=1}^{\infty}[(1-q^j)(1-e^{2\pi\sqrt{-1}v}q^j)(1-e^{-2\pi\sqrt{-1}v}q^j)],
   \end{equation}
\begin{equation}
    \theta_1(v,\tau)=2q^{\frac{1}{8}}{\rm cos}(\pi
   v)\prod_{j=1}^{\infty}[(1-q^j)(1+e^{2\pi\sqrt{-1}v}q^j)(1+e^{-2\pi\sqrt{-1}v}q^j)],
\end{equation}
\begin{equation}
    \theta_2(v,\tau)=\prod_{j=1}^{\infty}[(1-q^j)(1-e^{2\pi\sqrt{-1}v}q^{j-\frac{1}{2}})
(1-e^{-2\pi\sqrt{-1}v}q^{j-\frac{1}{2}})],
\end{equation}
\begin{equation}
   \theta_3(v,\tau)=\prod_{j=1}^{\infty}[(1-q^j)(1+e^{2\pi\sqrt{-1}v}q^{j-\frac{1}{2}})
(1+e^{-2\pi\sqrt{-1}v}q^{j-\frac{1}{2}})],
\end{equation}
 \noindent
where $q=e^{2\pi\sqrt{-1}\tau}$ with $\tau\in\textbf{H}$, the upper
half complex plane. Let
\begin{equation}
    \theta'(0,\tau)=\frac{\partial\theta(v,\tau)}{\partial v}|_{v=0}.
\end{equation} \noindent Then the following Jacobi identity
(cf. \cite{Ch}) holds,
\begin{equation}   \theta'(0,\tau)=\pi\theta_1(0,\tau)\theta_2(0,\tau)\theta_3(0,\tau).
\end{equation}
\noindent Denote $$SL_2({\bf Z})=\left\{\left(\begin{array}{cc}
\ a & b  \\
 c  & d
\end{array}\right)\mid a,b,c,d \in {\bf Z},~ad-bc=1\right\}$$ the
modular group. Let $S=\left(\begin{array}{cc}
\ 0 & -1  \\
 1  & 0
\end{array}\right),~T=\left(\begin{array}{cc}
\ 1 &  1 \\
 0  & 1
\end{array}\right)$ be the two generators of $SL_2(\bf{Z})$. They
act on $\textbf{H}$ by $S\tau=-\frac{1}{\tau},~T\tau=\tau+1$. One
has the following transformation laws of theta functions under the
actions of $S$ and $T$ (cf. \cite{Ch}):
\begin{equation}
    \theta(v,\tau+1)=e^{\frac{\pi\sqrt{-1}}{4}}\theta(v,\tau),~~\theta(v,-\frac{1}{\tau})
=\frac{1}{\sqrt{-1}}\left(\frac{\tau}{\sqrt{-1}}\right)^{\frac{1}{2}}e^{\pi\sqrt{-1}\tau
v^2}\theta(\tau v,\tau);
\end{equation}
 \begin{equation}
     \theta_1(v,\tau+1)=e^{\frac{\pi\sqrt{-1}}{4}}\theta_1(v,\tau),~~\theta_1(v,-\frac{1}{\tau})
=\left(\frac{\tau}{\sqrt{-1}}\right)^{\frac{1}{2}}e^{\pi\sqrt{-1}\tau
v^2}\theta_2(\tau v,\tau);
 \end{equation}
\begin{equation}   \theta_2(v,\tau+1)=\theta_3(v,\tau),~~\theta_2(v,-\frac{1}{\tau})
=\left(\frac{\tau}{\sqrt{-1}}\right)^{\frac{1}{2}}e^{\pi\sqrt{-1}\tau
v^2}\theta_1(\tau v,\tau);
\end{equation}
\begin{equation}    \theta_3(v,\tau+1)=\theta_2(v,\tau),~~\theta_3(v,-\frac{1}{\tau})
=\left(\frac{\tau}{\sqrt{-1}}\right)^{\frac{1}{2}}e^{\pi\sqrt{-1}\tau
v^2}\theta_3(\tau v,\tau).
\end{equation}
\begin{equation}
 \theta'(0,-\frac{1}{\tau})=\frac{1}{\sqrt{-1}}\left(\frac{\tau}{\sqrt{-1}}\right)^{\frac{1}{2}}
\tau\theta'(0,\tau),
\end{equation}

\begin{equation}
    \theta'(v,\tau+1)=e^{\frac{\pi\sqrt{-1}}{4}}\theta'(v,\tau),
 \end{equation}
 \begin{equation}
    ~~\theta'(v,-\frac{1}{\tau})
=\frac{1}{\sqrt{-1}}\left(\frac{\tau}{\sqrt{-1}}\right)^{\frac{1}{2}}e^{\pi\sqrt{-1}\tau
v^2}(2\pi \sqrt{-1}\tau\theta(\tau v,\tau)+\tau\theta'(\tau v,\tau));
\end{equation}
 \begin{equation}
    \theta'_1(v,\tau+1)=e^{\frac{\pi\sqrt{-1}}{4}}\theta'_1(v,\tau),
 \end{equation}
 \begin{equation}
    ~~\theta'_1(v,-\frac{1}{\tau})
=\frac{1}{\sqrt{-1}}\left(\frac{\tau}{\sqrt{-1}}\right)^{\frac{1}{2}}e^{\pi\sqrt{-1}\tau
v^2}(2\pi \sqrt{-1}\tau\theta_2(\tau v,\tau)+\tau\theta'_2(\tau v,\tau));
\end{equation}
\begin{equation}
    \theta'_2(v,\tau+1)=e^{\frac{\pi\sqrt{-1}}{4}}\theta'_3(v,\tau),
 \end{equation}
 \begin{equation}
    ~~\theta'_2(v,-\frac{1}{\tau})
=\frac{1}{\sqrt{-1}}\left(\frac{\tau}{\sqrt{-1}}\right)^{\frac{1}{2}}e^{\pi\sqrt{-1}\tau
v^2}(2\pi \sqrt{-1}\tau\theta_1(\tau v,\tau)+\tau\theta'_1(\tau v,\tau));
\end{equation}
\begin{equation}
    \theta'_3(v,\tau+1)=e^{\frac{\pi\sqrt{-1}}{4}}\theta'_2(v,\tau),
 \end{equation}
 \begin{equation}
    ~~\theta'_3(v,-\frac{1}{\tau})
=\frac{1}{\sqrt{-1}}\left(\frac{\tau}{\sqrt{-1}}\right)^{\frac{1}{2}}e^{\pi\sqrt{-1}\tau
v^2}(2\pi \sqrt{-1}\tau\theta_3(\tau v,\tau)+\tau\theta'_3(\tau v,\tau));
\end{equation}
\noindent
 \noindent {\bf Definition 2.1} A modular form over $\Gamma$, a
 subgroup of $SL_2({\bf Z})$, is a holomorphic function $f(\tau)$ on
 $\textbf{H}$ such that
 \begin{equation}
    f(g\tau):=f\left(\frac{a\tau+b}{c\tau+d}\right)=\chi(g)(c\tau+d)^kf(\tau),
 ~~\forall g=\left(\begin{array}{cc}
\ a & b  \\
 c & d
\end{array}\right)\in\Gamma,
 \end{equation}
\noindent where $\chi:\Gamma\rightarrow {\bf C}^{\star}$ is a
character of $\Gamma$. $k$ is called the weight of $f$.\\
Let $$\Gamma_0(2)=\left\{\left(\begin{array}{cc}
\ a & b  \\
 c  & d
\end{array}\right)\in SL_2({\bf Z})\mid c\equiv 0~({\rm
mod}~2)\right\},$$
$$\Gamma^0(2)=\left\{\left(\begin{array}{cc}
\ a & b  \\
 c  & d
\end{array}\right)\in SL_2({\bf Z})\mid b\equiv 0~({\rm
mod}~2)\right\},$$
 $$\Gamma_{\theta}=\left\{\left(\begin{array}{cc}
\ a & b  \\
 c  & d
\end{array}\right)\in SL_2({\bf Z})\mid \left(\begin{array}{cc}
\ a & b  \\
 c  & d
\end{array}\right)
\equiv \left(\begin{array}{cc}
\ 1 & 0  \\
 0 & 1
\end{array}\right) ~{\rm or}~
 \left(\begin{array}{cc}
\ 0 & 1  \\
 1  & 0
\end{array}\right)
 ~({\rm
mod}~2)\right\},$$
be the three modular subgroups of $SL_2({\bf Z})$.
It is known that the generators of $\Gamma_0(2)$ are $T,~ST^2ST$,
the generators of $\Gamma^0(2)$ are $STS,~T^2STS$, the generators of $\Gamma_{\theta}$ are $S,~T^2$ (cf. \cite{Ch}).\\
\indent If $\Gamma$ is a modular subgroup, let ${\mathcal{M}}_{{\bf
R}}(\Gamma)$ denote the ring of modular forms over $\Gamma$ with
real Fourier coefficients. Writing $\theta_j=\theta_j(0,\tau),~1\leq
j\leq 3,$ we introduce six explicit modular forms (cf. \cite{Li1}),
$$\delta_1(\tau)=\frac{1}{8}(\theta_2^4+\theta_3^4),~~\varepsilon_1(\tau)=\frac{1}{16}\theta_2^4\theta_3^4,$$
$$8\delta_2(\tau)=-\frac{1}{8}(\theta_1^4+\theta_3^4),~~\varepsilon_2(\tau)=\frac{1}{16}\theta_1^4\theta_3^4,$$
\noindent They have the following Fourier expansions in
$q^{\frac{1}{2}}$:
$$\delta_1(\tau)=\frac{1}{4}+6q+\cdots,~~\varepsilon_1(\tau)=\frac{1}{16}-q+\cdots,$$
$$8\delta_2(\tau)=-1-24q^{\frac{1}{2}}-24q+\cdots,~~\varepsilon_2(\tau)=q^{\frac{1}{2}}+8q+\cdots,$$
\noindent where the $"\cdots"$ terms are the higher degree terms,
all of which have integral coefficients. They also satisfy the
transformation laws,
\begin{equation}
    \delta_2(-\frac{1}{\tau})=\tau^2\delta_1(\tau),~~~~~~\varepsilon_2(-\frac{1}{\tau})
=\tau^4\varepsilon_1(\tau),
\end{equation}
\noindent {\bf Lemma 2.2} (\cite{Li1}) {\it $\delta_1(\tau)$ (resp.
$\varepsilon_1(\tau)$) is a modular form of weight $2$ (resp. $4$)
over $\Gamma_0(2)$, $\delta_2(\tau)$ (resp. $\varepsilon_2(\tau)$)
is a modular form of weight $2$ (resp. $4$) over $\Gamma^0(2)$,
while  $\delta_3(\tau)$ (resp. $\varepsilon_3(\tau)$) is a modular
form of weight $2$ (resp. $4$) over $\Gamma_\theta(2)$ and moreover
${\mathcal{M}}_{{\bf R}}(\Gamma^0(2))={\bf
R}[\delta_2(\tau),\varepsilon_2(\tau)]$.}\\

We recall the Eisenstein series $G_{2k}(\tau)$ are defined to be
  \begin{equation}
  E_{2k}(\tau):=1-\frac{4k}{B_{2k}}\sum_{n=1}^{\infty}\sigma_{2k-1}(n)\cdot q^n,~~G_{2k}(\tau)=-\frac{B_{2k}}{4k}\cdot E_{2k}(\tau)
\end{equation}
where $\sigma_k(n):=\sum_{m>0,~m|n}m^k$ and $B_{2k}$ are the Bernoulli numbers. It is well known that the whole grading ring of modular forms
over $SL(2,Z)$ are generated by $E_4(\tau)$ and $E_6(\tau)$ and
 \begin{equation}
  E_{4}(\tau):=1+240q+2160q^2+\cdots,~~E_{6}(\tau)=1-504q-16632q^2+\cdots.
\end{equation}
\\
\noindent{\bf Definition 2.3} A meromorphic Jacobi form of index $m$ and weight $l$ over $L\times \Gamma$, where $L$ is an integral lattice in the complex plane $\mathbb{C}$ preserved by the modular subgroup $\Gamma\subset SL(2,Z)$, is a (meromorphic) function $F(z,\tau)$ on $C\times H$ such that
\begin{align}
F(\frac{z}{c\tau+d_0},\frac{a\tau+b}{c\tau+d_0})=(c\tau+d_0)^{l}{\rm exp}(2\pi\sqrt{-1}m\frac{cz^2}{c\tau+d_0}){F}(z,\tau),
\end{align}
\begin{align}
{F}(z+\lambda \tau+\mu,\tau)={\rm exp}(-2\pi\sqrt{-1}m(2\lambda z+\lambda^2\tau))
{F}(z,\tau),
\end{align}
where $\left(\begin{array}{cc}
\ a & b  \\
 c  & d_0
\end{array}\right)\in \Gamma$ and $\lambda,\mu \in L$.\\

\section{The two-variable elliptic genera which are holomorphic $SL(2,Z)$-Jacobi forms. }
\noindent {\bf 3.1 The first two-variable elliptic genus which is holomorphic $SL(2,Z)$-Jacobi forms}

\indent Let $M$ be a $2d$-dimensional closed ${\rm spin}$ manifold. Let $W$ denote a real $2l$-dimensional spin vector bundle on $M$ with the almost complex structure $J$ and $\triangle(W)$ be the associated spinor bundle, then $W\otimes C=W^{1,0}\oplus W^{0,1}$. Denote the first Chern classes of $W^{1,0}$ by $c_1(W^{1,0})$. We denote by
$\pm 2\pi \sqrt{-1}x_i~(1\leq i\leq d)$ and $2\pi \sqrt{-1}w_j~(1\leq j\leq l)$ and $-2\pi \sqrt{-1}w_j~(1\leq j\leq l)$ respectively the formal Chern roots of $TM\otimes C$ and $W^{1,0}$ and $W^{0,1}$.
Let $(\tau,z)\in \mathcal{H}\times \mathcal{C}$ where $\mathcal{H}$ is the upper half plane and $\mathcal{C}$ is the complex plane. Let $y=e^{2\pi \sqrt{-1}z}$ and $q= e^{2\pi \sqrt{-1}\tau}$
\\

Let \begin{align}
&Q_1:=\triangle(W)\bigotimes _{m=1}^{\infty}
\left( \wedge_{yq^{m}}(W^{1,0})\otimes\wedge_{y^{-1}q^{m}}(W^{0,1})\otimes \wedge_{q^{m}}(-C^{2l})\right)\\\notag
&Q_2:=\bigotimes _{m=1}^{\infty}
\left( \wedge_{-yq^{m-\frac{1}{2}}}(W^{1,0})\otimes\wedge_{-y^{-1}q^{m-\frac{1}{2}}}(W^{0,1})\otimes\wedge_{-q^{m-\frac{1}{2}}}(-C^{2l})\right)\\\notag
&Q_3:=\bigotimes _{m=1}^{\infty}
\left( \wedge_{yq^{m-\frac{1}{2}}}(W^{1,0})\otimes\wedge_{y^{-1}q^{m-\frac{1}{2}}}(W^{0,1})\otimes\wedge_{q^{m-\frac{1}{2}}}(-C^{2l})\right)
\end{align}

\begin{defn}
The two-variable elliptic genus of $M^{2d}$ with respect to $W$, which we denote by ${\rm Ell}(M,W,\tau,z)$ is defined by
\begin{equation}
{\rm Ell}(M,W,\tau,z):=\int_M \widehat{A}(M){\rm ch}(\bigotimes _{n=1}^{\infty}S_{q^n}(\widetilde{TM\otimes C}))
{\rm ch}(Q_1\otimes Q_2\otimes Q_3),
\end{equation}
which is the index of the twisted Dirac operator.
\end{defn}
By (2.4)-(2.9), we have

\begin{lem} The following equality holds:
\begin{align}
{\rm Ell}(M,W,\tau,z)=2^l\int_M\prod_{i=1}^d\frac{x_i\theta'(0,\tau)}{\theta(x_i,\tau)}\prod_{\alpha=1}^l\frac{\theta_1
(\tau,w_\alpha+z)}{\theta_1(0,\tau)}\frac{\theta_2
(\tau,w_\alpha+z)}{\theta_2(0,\tau)}\frac{\theta_3
(\tau,w_\alpha+z)}{\theta_3(0,\tau)}.
\end{align}
\end{lem}
By (2.12)-(2.24), we have
\begin{thm}
If $c_1(W^{1,0})=0$ and the first Pontrjagin classes $p_1(M)=3p_1(W)$, then
\begin{align}
{\rm Ell}(M,W,\frac{z}{c\tau+d_0},\frac{a\tau+b}{c\tau+d_0})=(c\tau+d_0)^{d}{\rm exp}(3\pi\sqrt{-1}l\frac{cz^2}{c\tau+d_0}){\rm Ell}(M,W,z,\tau),
\end{align}
\begin{align}
{\rm Ell}(M,W,z+\lambda \tau+\mu,\tau)=(-1)^{l(\mu +\lambda) }{\rm exp}(-3\pi\sqrt{-1}l(2\lambda z+\lambda^2\tau))
{\rm Ell}(M,W,z,\tau).
\end{align}
When $l$ is even, the elliptic genus ${\rm Ell}(M,W,\tau,z)$ is a weak $SL(2,Z)$-Jacobi form of weight $d$ and index $\frac{3l}{2}$.
\end{thm}

We recall Proposition 3.5 in \cite{LiP}.
\begin{prop}(\cite{LiP}) Suppose a function $\varphi(\tau,z):\mathbb{H}\times \mathbb{C}\rightarrow \mathbb{C}$ satisfies
\begin{align}
\varphi(\frac{a\tau+b}{c\tau+d_0},\frac{z}{c\tau+d_0})=(c\tau+d_0)^{k}{\rm exp}(\frac{2\pi\sqrt{-1}mcz^2}{c\tau+d_0})\varphi(\tau,z);~~
\left(\begin{array}{cc}
\ a & b  \\
 c  & d_0
\end{array}\right)\in SL(2,Z).
\end{align}
We define
\begin{align}
&\Phi(\tau,z):={\rm exp}(-8\pi^2mG_2(\tau)z^2)\varphi(\tau,z):=\sum_{n\geq 0}a_n(\tau)\cdot z^n,
\end{align}
then these $a_n(\tau)$ are modular forms of weight $k+n$ over $SL(2,Z)$.
\end{prop}

\begin{prop}If $c_1(W^{1,0})=0$ and the first Pontrjagin classes $p_1(M)=3p_1(W)$, then the series $a_n(M,W,\tau)$ determined by
\begin{align}
{\rm exp}(-12\pi^2lG_2(\tau)z^2){\rm Ell}(M,W,\tau,z)=\sum_{n\geq 0}a_n(M,W£¬\tau)\cdot z^n,
\end{align}
are modular forms of weight $d+n$ over $SL(2,Z)$. Furthermore, the first series of $a_n(M,W,\tau):=a_n^0+a_n^1q+a_n^2q^2+\cdots$ are of the following form:
\begin{align}
&a_0^0=\int_M\widehat{A}(TM){\rm ch}(\triangle(W)),\\\notag
&a_0^1=
\int_M\widehat{A}(TM){\rm ch}(\triangle(W)){\rm ch}(\widetilde{T_CM}+\widetilde{W\otimes C}
+2\wedge^2\widetilde{W\otimes C}-\widetilde{W\otimes C}\otimes \widetilde{W\otimes C}),\\\notag
&a_0^2=
\int_M\widehat{A}(TM){\rm ch}(\triangle(W)){\rm ch}(A_0),
\end{align}
where
\begin{align}
A_0&=2\wedge^4\widetilde{W\otimes C}+\wedge^2\widetilde{W\otimes C}\otimes \wedge^2\widetilde{W\otimes C}
-2\widetilde{W\otimes C}\otimes\wedge^3\widetilde{W\otimes C}\\\notag
&+(\widetilde{T_CM}+\widetilde{W\otimes C})\otimes(
2\wedge^2\widetilde{W\otimes C}-\widetilde{W\otimes C}\otimes \widetilde{W\otimes C})\\\notag
&+\wedge^2\widetilde{W\otimes C}
+S^2\widetilde{T_CM}+\widetilde{T_CM}+\widetilde{W\otimes C}.
\end{align}
\begin{align}
&a_1^0=a_3^0=0,~a_2^0=\frac{l}{2}\pi^2\int_M\widehat{A}(TM){\rm ch}(\triangle(W)),
~a_4^0=\frac{4}{3}\pi^4l^2\int_M\widehat{A}(TM){\rm ch}(\triangle(W)).
\end{align}
\begin{align}
&a^1_1=2\pi\sqrt{-1}\int_M\widehat{A}(TM){\rm ch}(\triangle(W)){\rm ch}(W^{1,0}-W^{0,1}+4\wedge^2W^{1,0}-2W^{1,0}\otimes W^{1,0}\\\notag
&-4\wedge^2W^{0,1}
-2W^{0,1}\otimes W^{0,1}).
\end{align}
\end{prop}

\begin{proof} We know that ${\rm Ell}(M,W,0,\tau)=a_0(M,W,\tau)$ and
\begin{equation}
{\rm Ell}(M,W,0,\tau)=\int_M \widehat{A}(M){\rm ch}(\bigotimes _{n=1}^{\infty}S_{q^n}(\widetilde{TM\otimes C}))
{\rm ch}(Q_1\otimes Q_2\otimes Q_3)|_{y=1}¡£
\end{equation}
So we get (3.9). If we set
\begin{equation}{\rm exp}(-12\pi^2lG_2(\tau)z^2):=A_0(z)+A_1(z)q+O(q^2),
\end{equation}
and
\begin{equation}{\rm Ell}(TM,W,\tau,z):=B_0(z)+B_1(z)q+O(q^2),
\end{equation}
we can get that
\begin{align}
&A_0(z)=1+\frac{l}{2}\pi^2z^2+\frac{3l^2}{4}\pi^4z^4+O(z^6),\\\notag
&A_1(z)=-12l\pi^2z^2-6\pi^4l^2z^4+O(z^6),\\\notag
&B_0(z)=\int_M\widehat{A}(TM){\rm ch}(\triangle(W)),\\\notag
&B_1(z)=\int_M\widehat{A}(TM){\rm ch}(\widetilde{T_CM}\otimes\triangle(W))
+\int_M\widehat{A}(TM){\rm ch}(\triangle(W))\\\notag
&\otimes(yW^{1,0}+y^{-1}W^{0,1}+2y^2\wedge^2W^{1,0}-y^2W^{1,0}\otimes W^{1,0}+2y^{-2}\wedge^2W^{0,1}-y^{-2}W^{0,1}\otimes W^{0,1}).
\end{align}
We know that
\begin{equation}
\sum_{n\geq 0}a_n(TM,W,\tau)\cdot z^n=A_0(z)B_0(z)+[A_0(z)B_1(z)+A_1(z)B_0(z)]q+\cdots,
\end{equation}
then we can get Proposition 3.5 by (3.14)-(3.17).
\end{proof}

Similar to Theorem 3.4 in \cite{Wa1}, we have

\begin{thm}Let $c_1(W^{1,0})=0$ and the first Pontrjagin classes $p_1(M)=3p_1(W)$, then\\
1)if $d=4$, then
\begin{align}
&\int_M\widehat{A}(TM){\rm ch}(\triangle(W)){\rm ch}(\widetilde{T_CM}+\widetilde{W\otimes C}
+2\wedge^2\widetilde{W\otimes C}-\widetilde{W\otimes C}\otimes \widetilde{W\otimes C})\\\notag
&=240\int_M\widehat{A}(TM){\rm ch}(\triangle(W))
\end{align}
and
so $a_0^1$ is the integer multiple of $240$. By the Atiyah-Patodi-Singer index theorem, when $M$ is an $8$-dimensional spin manifold with boundary and has the product structure near the boundary, then
\begin{align}
&{\rm Ind}(D^+_{\triangle(W)\otimes F})
\equiv 240\widetilde{\eta}(D_{\partial M, \triangle(W)})-\widetilde{\eta}(D_{\partial Z}\otimes
(\triangle(W)\otimes F)),~~~{\rm mod}~~240,
\end{align}
where $\widetilde{\eta}(D_{\partial M})$ is a reduced eta invariant and
\begin{align}
F:=\widetilde{T_CM}+\widetilde{W\otimes C}
+2\wedge^2\widetilde{W\otimes C}-\widetilde{W\otimes C}\otimes \widetilde{W\otimes C}.
\end{align}
\\
2)if $d=6$, then
\begin{align}
&\int_M\widehat{A}(TM){\rm ch}(\triangle(W)){\rm ch}(\widetilde{T_CM}+\widetilde{W\otimes C}
+2\wedge^2\widetilde{W\otimes C}-\widetilde{W\otimes C}\otimes \widetilde{W\otimes C})\\\notag
&=-504\int_M\widehat{A}(TM){\rm ch}(\triangle(W))
\end{align}
and
so $a_0^1$ is the integer multiple of $504$. When $M$ is an $12$-dimensional spin manifold with boundary and has the product structure near the boundary, then
\begin{align}
&{\rm Ind}(D^+_{\triangle(W)\otimes F})
\equiv -504\widetilde{\eta}(D_{\partial M, \triangle(W)})-\widetilde{\eta}(D_{\partial Z}\otimes
(\triangle(W)\otimes F)),~~~{\rm mod}~~504.
\end{align}
\\

3)if $d=8$, then
\begin{align}
&\int_M\widehat{A}(TM){\rm ch}(\triangle(W)){\rm ch}(\widetilde{T_CM}+\widetilde{W\otimes C}
+2\wedge^2\widetilde{W\otimes C}-\widetilde{W\otimes C}\otimes \widetilde{W\otimes C})\\\notag
&=480\int_M\widehat{A}(TM){\rm ch}(\triangle(W))
\end{align}
and
so $a_0^1$ is the integer multiple of $480$. When $M$ is an $16$-dimensional spin manifold with boundary and has the product structure near the boundary, then
\begin{align}
&{\rm Ind}(D^+_{\triangle(W)\otimes F})
\equiv 480\widetilde{\eta}(D_{\partial M, \triangle(W)})-\widetilde{\eta}(D_{\partial Z}\otimes
(\triangle(W)\otimes F)),~~~{\rm mod}~~480.
\end{align}
\\

4)if $d=10$, then
\begin{align}
&\int_M\widehat{A}(TM){\rm ch}(\triangle(W)){\rm ch}(\widetilde{T_CM}+\widetilde{W\otimes C}
+2\wedge^2\widetilde{W\otimes C}-\widetilde{W\otimes C}\otimes \widetilde{W\otimes C})\\\notag
&=-264\int_M\widehat{A}(TM){\rm ch}(\triangle(W))
\end{align}
and
so $a_0^1$ is the integer multiple of $264$. When $M$ is an $20$-dimensional spin manifold with boundary and has the product structure near the boundary, then
\begin{align}
&{\rm Ind}(D^+_{\triangle(W)\otimes F})
\equiv -264\widetilde{\eta}(D_{\partial M, \triangle(W)})-\widetilde{\eta}(D_{\partial Z}\otimes
(\triangle(W)\otimes F)),~~~{\rm mod}~~264.
\end{align}
\\
5)if $d=14$, then
\begin{align}
&\int_M\widehat{A}(TM){\rm ch}(\triangle(W)){\rm ch}(\widetilde{T_CM}+\widetilde{W\otimes C}
+2\wedge^2\widetilde{W\otimes C}-\widetilde{W\otimes C}\otimes \widetilde{W\otimes C})\\\notag
&=-24\int_M\widehat{A}(TM){\rm ch}(\triangle(W))
\end{align}
and
so $a_0^1$ is the integer multiple of $24$. When $M$ is an $28$-dimensional spin manifold with boundary and has the product structure near the boundary, then
\begin{align}
&{\rm Ind}(D^+_{\triangle(W)\otimes F})
\equiv -24\widetilde{\eta}(D_{\partial M, \triangle(W)})-\widetilde{\eta}(D_{\partial Z}\otimes
(\triangle(W)\otimes F)),~~~{\rm mod}~~24.
\end{align}
\end{thm}

\begin{thm}Let $c_1(W^{1,0})=0$ and the first Pontrjagin classes $p_1(M)=3p_1(W)$, then\\
when $d$ is even or $d=3,~5,~7,~9,~13$, then
\begin{align}
&a^1_1=2\pi\sqrt{-1}\int_M\widehat{A}(TM){\rm ch}(\triangle(W)){\rm ch}(W^{1,0}-W^{0,1}+4\wedge^2W^{1,0}-2W^{1,0}\otimes W^{1,0}\\\notag
&-4\wedge^2W^{0,1}
-2W^{0,1}\otimes W^{0,1})=0.
\end{align}
\end{thm}
\begin{proof}:
When $d$ is even, then $a_1$ is a modular form of weight odd $d+1$ over $SL(2,Z)$. So $a_1=0$, then $a^1_1=0$. When $d=3,~5,~7,~9,~13$,
by Prop. 3.5, we get $a^1_1=c_0a^0_1$ where $c_0$ is a constant. By $a^0_1=0$, so $a^1_1=0$.\\
\end{proof}

\begin{lem}(\cite{BG} Lemma 6.1) Let $f$ be a modular form of weight $d$ over $SL(2,Z)$, if 1) $2<d<12$ and $f=O(q)$ i.e. ${\rm ord}_{\infty}(f)\geq 1$ or 2) $d=14$ and $f=O(q)$, then $f=0$.
\end{lem}
By the above Lemma and Prop. 3.5, we have:
\begin{prop}1)Let $c_1(W^{1,0})=0$ and the first Pontrjagin classes $p_1(M)=3p_1(W)$.
1)When $2<d<12$ or $d=14$, then $a_0^0=\int_M\widehat{A}(TM){\rm ch}(\triangle(W))=0$ if and only if $a_0=0$.\\
2)When $1<d<11$ or $d=13$, then $a_1=0$.
\end{prop}
~\\
\noindent {\bf 3.2 The second two-variable elliptic genus which is holomorphic $SL(2,Z)$-Jacobi forms}\\

\indent In this subsection, the fundamental setting is the same as the subsection 3.1.
\begin{defn}
The two-variable elliptic genus of $M^{2d}$ with respect to $W$, which we denote by $\widehat{{\rm Ell}}(M,W,\tau,z)$ is defined by
\begin{equation}
\widehat{{\rm Ell}}(M,W,\tau,z):=\int_M \widehat{A}(M){\rm ch}(\bigotimes _{n=1}^{\infty}S_{q^n}(\widetilde{TM\otimes C}))
({\rm ch}(Q_1)+2^l{\rm ch}(Q_1)+2^l{\rm ch}(Q_2)),
\end{equation}
which is the index of the twisted Dirac operator.
\end{defn}
By (2.4)-(2.9), we have

\begin{lem} The following equality holds:
\begin{align}
\widehat{{\rm Ell}}(M,W,\tau,z)=2^l\int_M\prod_{i=1}^d\frac{x_i\theta'(0,\tau)}{\theta(x_i,\tau)}\left[\prod_{\alpha=1}^l\frac{\theta_1
(\tau,w_\alpha+z)}{\theta_1(0,\tau)}+\prod_{\alpha=1}^l\frac{\theta_2
(\tau,w_\alpha+z)}{\theta_2(0,\tau)}+\prod_{\alpha=1}^l\frac{\theta_3
(\tau,w_\alpha+z)}{\theta_3(0,\tau)}\right]
\end{align}
\end{lem}
By (2.12)-(2.24), we have
\begin{thm}
If $c_1(W^{1,0})=0$ and the first Pontrjagin classes $p_1(M)=p_1(W)$, then
\begin{align}
\widehat{{\rm Ell}}(M,W,\frac{z}{c\tau+d_0},\frac{a\tau+b}{c\tau+d_0})=(c\tau+d_0)^{d}{\rm exp}(\pi\sqrt{-1}l\frac{cz^2}{c\tau+d_0})\widehat{{\rm Ell}}(M,W,z,\tau).
\end{align}
When $l$ is even, the elliptic genus $\widehat{{\rm Ell}}(M,W,\tau,z)$ is a weak $SL(2,Z)$-Jacobi form of weight $d$ and index $\frac{l}{2}$.
\end{thm}

By Proposition 3.4, we have

\begin{prop}If $c_1(W^{1,0})=0$ and the first Pontrjagin classes $p_1(M)=p_1(W)$, then the series $\widehat{a_n}(M,W,\tau)$ determined by
\begin{align}
{\rm exp}(-4\pi^2lG_2(\tau)z^2)\widehat{{\rm Ell}}(M,W,\tau,z)=\sum_{n\geq 0}\widehat{a_n}(M,W£¬\tau)\cdot z^n,
\end{align}
are modular forms of weight $d+n$ over $SL(2,Z)$. Furthermore, the first series of $\widehat{a_n}(M,W,\tau):=\widehat{a_n^0}+\widehat{a_n^1}q+\widehat{a_n^2}q^2+\cdots$ are of the following form:
\begin{align}
&\widehat{a_0^0}=\int_M\widehat{A}(TM){\rm ch}(\triangle(W))+2^{l+1}\int_M\widehat{A}(TM),\\\notag
&\widehat{a_0^1}=
\int_M\widehat{A}(TM){\rm ch}(\triangle(W)){\rm ch}(\widetilde{T_CM}+\widetilde{W\otimes C})+2^{l+1}
\int_M\widehat{A}(TM){\rm ch}(\widetilde{T_CM}+\wedge^2\widetilde{W\otimes C})\\\notag
&a_0^2=
\int_M\widehat{A}(TM){\rm ch}(\triangle(W)){\rm ch}(S^2\widetilde{T_CM}+\wedge^2\widetilde{W\otimes C}
+\widetilde{T_CM}\otimes\widetilde{W\otimes C}+\widetilde{T_CM}+\widetilde{W\otimes C})\\\notag
&+2^{l+1}\int_M\widehat{A}(TM){\rm ch}(S^2\widetilde{T_CM}+\widetilde{T_CM}+\widetilde{T_CM}\otimes \wedge^2\widetilde{W\otimes C}
+\wedge^4\widetilde{W\otimes C}+\widetilde{W\otimes C}\otimes \widetilde{W\otimes C}.
\end{align}

\begin{align}
&\widehat{a_1^0}=\widehat{a_3^0}=0,~\widehat{a_2^0}=\frac{l}{6}\pi^2(\int_M\widehat{A}(TM){\rm ch}(\triangle(W))+2^{l+1}\int_M\widehat{A}(TM)),\\\notag
&\widehat{a_4^0}=\frac{l^2}{12}\pi^4(\int_M\widehat{A}(TM){\rm ch}(\triangle(W))+2^{l+1}\int_M\widehat{A}(TM)).
\end{align}
\begin{align}
&\widehat{a^1_1}=2\pi\sqrt{-1}[\int_M\widehat{A}(TM){\rm ch}(\triangle(W)){\rm ch}(W^{1,0}-W^{0,1})\\\notag
&+2^{l+1}\int_M\widehat{A}(TM){\rm ch}(
2\wedge^2W^{1,0}+2lW^{1,0}-2l W^{1,0}
-2\wedge^2W^{0,1})].
\end{align}
\end{prop}

\begin{proof} We know that $\widehat{{\rm Ell}}(M,W,0,\tau)=\widehat{a_0}(M,W,\tau)$ and
\begin{equation}
\widehat{{\rm Ell}}(M,W,0,\tau)=\int_M \widehat{A}(M){\rm ch}(\bigotimes _{n=1}^{\infty}S_{q^n}(\widetilde{TM\otimes C}))
{\rm ch}(Q_1+2^l Q_2+2^lQ_3)|_{y=1}.
\end{equation}
So we get (3.34). If we set
\begin{equation}{\rm exp}(-4\pi^2lG_2(\tau)z^2):=\widehat{A_0(Z)}+\widehat{A_1(z)}q+O(q^2),
\end{equation}
and
\begin{equation}\widehat{{\rm Ell}}(TM,W,\tau,z):=\widehat{B_0(Z)}+\widehat{B_1(z)}q+O(q^2),
\end{equation}
we can get that
\begin{align}
&\widehat{A_0(z)}=1+\frac{l}{6}\pi^2z^2+\frac{l^2}{12}\pi^4z^4+O(z^6),\\\notag
&\widehat{A_1(z)}=-4l\pi^2z^2-\frac{2}{3}\pi^4l^2z^4+O(z^6),\\\notag
&\widehat{B_0(z)}=\int_M\widehat{A}(TM){\rm ch}(\triangle(W))+2^{l+1}\int_M\widehat{A}(TM),\\\notag
&\widehat{B_1(z)}=\int_M\widehat{A}(TM){\rm ch}(\triangle(W)\otimes\widetilde{T_CM})+2^{l+1}\int_M\widehat{A}(TM){\rm ch}(\widetilde{T_CM})\\\notag
&+\int_M\widehat{A}(TM){\rm ch}(\triangle(W)\otimes(yW^{1,0}+y^{-1}W^{0,1}-2l)\\\notag
&+2^{l+1}\int_M\widehat{A}(TM){\rm ch}(
y^{-2}\wedge^2W^{0,1}+W^{1,0}\otimes W^{0,1}+y^2\wedge^2W^{1,0}-2ly^{-1}W^{0,1}-2lyW^{1,0}).
\end{align}
We know that
\begin{equation}
\sum_{n\geq 0}\widehat{a_n}(TM,W,\tau)\cdot z^n=\widehat{A}_0(z)\widehat{B}_0(z)+[\widehat{A}_0(z)\widehat{B}_1(z)+\widehat{A}_1(z)\widehat{B}_0(z)]q+\cdots,
\end{equation}
then we can get Proposition 3.13 by (3.38)-(3.41).
\end{proof}

Similar to Theorem 3.6 in \cite{Wa1}, we have

\begin{thm}Let $c_1(W^{1,0})=0$ and the first Pontrjagin classes $p_1(M)=p_1(W)$, then\\
1)if $d=4$, then
\begin{align}
&\int_M\widehat{A}(TM){\rm ch}(\triangle(W)){\rm ch}(\widetilde{T_CM}+\widetilde{W\otimes C})+2^{l+1}
\int_M\widehat{A}(TM){\rm ch}(\widetilde{T_CM}+\wedge^2\widetilde{W\otimes C})\\\notag
&=240[\int_M\widehat{A}(TM){\rm ch}(\triangle(W))+2^{l+1}\int_M\widehat{A}(TM)]
\end{align}
and
so $\widehat{a_0^1}$ is the integer multiple of $240$.
\\
2)if $d=6$, then
\begin{align}
&\int_M\widehat{A}(TM){\rm ch}(\triangle(W)){\rm ch}(\widetilde{T_CM}+\widetilde{W\otimes C})+2^{l+1}
\int_M\widehat{A}(TM){\rm ch}(\widetilde{T_CM}+\wedge^2\widetilde{W\otimes C})\\\notag
&=-504[\int_M\widehat{A}(TM){\rm ch}(\triangle(W))+2^{l+1}\int_M\widehat{A}(TM)]
\end{align}
and
so $\widehat{a_0^1}$ is the integer multiple of $504$.
\\

3)if $d=8$, then
\begin{align}
&\int_M\widehat{A}(TM){\rm ch}(\triangle(W)){\rm ch}(\widetilde{T_CM}+\widetilde{W\otimes C})+2^{l+1}
\int_M\widehat{A}(TM){\rm ch}(\widetilde{T_CM}+\wedge^2\widetilde{W\otimes C})\\\notag
&=480[\int_M\widehat{A}(TM){\rm ch}(\triangle(W))+2^{l+1}\int_M\widehat{A}(TM)]
\end{align}
and
so $\widehat{a_0^1}$ is the integer multiple of $480$.
\\

4)if $d=10$, then
\begin{align}
&\int_M\widehat{A}(TM){\rm ch}(\triangle(W)){\rm ch}(\widetilde{T_CM}+\widetilde{W\otimes C})+2^{l+1}
\int_M\widehat{A}(TM){\rm ch}(\widetilde{T_CM}+\wedge^2\widetilde{W\otimes C})\\\notag
&=-264[\int_M\widehat{A}(TM){\rm ch}(\triangle(W))+2^{l+1}\int_M\widehat{A}(TM)]
\end{align}
and
so $\widehat{a_0^1}$ is the integer multiple of $264$.
\\
5)if $d=14$, then
\begin{align}
&\int_M\widehat{A}(TM){\rm ch}(\triangle(W)){\rm ch}(\widetilde{T_CM}+\widetilde{W\otimes C})+2^{l+1}
\int_M\widehat{A}(TM){\rm ch}(\widetilde{T_CM}+\wedge^2\widetilde{W\otimes C})\\\notag
&=-24[\int_M\widehat{A}(TM){\rm ch}(\triangle(W))+2^{l+1}\int_M\widehat{A}(TM)]
\end{align}
and
so $\widehat{a_0^1}$ is the integer multiple of $24$.
\end{thm}
Similar to Theorem 3.7, we have
\begin{thm}Let $c_1(W^{1,0})=0$ and the first Pontrjagin classes $p_1(M)=p_1(W)$, then\\
when $d$ is even or $d=3,~5,~7,~9,~13$, then
\begin{align}
&\widehat{a^1_1}=2\pi\sqrt{-1}[\int_M\widehat{A}(TM){\rm ch}(\triangle(W)){\rm ch}(W^{1,0}-W^{0,1})\\\notag
&+2^{l+1}\int_M\widehat{A}(TM){\rm ch}(
2\wedge^2W^{1,0}+2lW^{1,0}-2l W^{1,0}
-2\wedge^2W^{0,1})]=0.
\end{align}
\end{thm}

Similar to Prop. 3.9, we have:
\begin{prop}1)Let $c_1(W^{1,0})=0$ and the first Pontrjagin classes $p_1(M)=p_1(W)$.
1)When $2<d<12$ or $d=14$, then $\widehat{a_0^0}=\int_M\widehat{A}(TM){\rm ch}(\triangle(W))+2^{l+1}\int_M\widehat{A}(TM)=0$ if and only if $\widehat{a_0}=0$.\\
2)When $1<d<11$ or $d=13$, then $\widehat{a_1}=0$.
\end{prop}

\section{Two-variable elliptic genera which are $\Gamma_0(2)$, $\Gamma^0(2)$, $\Gamma _{\theta}$-Jacobi forms}

\indent In this section, the fundamental setting is the same as the section 3.
\begin{defn}
The two-variable elliptic genera of $M^{2d}$ with respect to $W$, which we denote by ${\rm Ell}_j(M,W,\tau,z)$ for $1\leq j\leq 3$ are defined by
\begin{align}
&{\rm Ell}_1(M,W,\tau,z):=\int_M \widehat{A}(M){\rm ch}(\bigotimes _{n=1}^{\infty}S_{q^n}(\widetilde{TM\otimes C}))
{\rm ch}(Q_1),\\\notag
&{\rm Ell}_j(M,W,\tau,z):=2^l\int_M \widehat{A}(M){\rm ch}(\bigotimes _{n=1}^{\infty}S_{q^n}(\widetilde{TM\otimes C}))
{\rm ch}(Q_j),~{\rm for}~j=2,~3,
\end{align}
which is the index of the twisted Dirac operator.
\end{defn}
By (2.4)-(2.9), we have

\begin{lem} The following equality holds:
\begin{align}
{\rm Ell}_j(M,W,\tau,z)=2^l\int_M\prod_{i=1}^d\frac{x_i\theta'(0,\tau)}{\theta(x_i,\tau)}\prod_{\alpha=1}^l\frac{\theta_j
(\tau,w_\alpha+z)}{\theta_j(0,\tau)},~1\leq j\leq 3.
\end{align}
\end{lem}
By (2.12)-(2.24), we have
\begin{thm}
If $c_1(W^{1,0})=0$ and the first Pontrjagin classes $p_1(M)=p_1(W)$, then for $1\leq j\leq 3$
\begin{align}
{\rm Ell}_j(M,W,\frac{z}{c\tau+d_0},\frac{a\tau+b}{c\tau+d_0})=(c\tau+d_0)^{d}{\rm exp}(\pi\sqrt{-1}l\frac{cz^2}{c\tau+d_0}){\rm Ell}_j(M,W,z,\tau),
\end{align}
\begin{align}
{\rm Ell}_1(M,W,z+\lambda \tau+\mu,\tau)=(-1)^{l\mu  }{\rm exp}(-\pi\sqrt{-1}l(2\lambda z+\lambda^2\tau))
{\rm Ell}_1(M,W,z,\tau),
\end{align}
\begin{align}
{\rm Ell}_2(M,W,z+\lambda \tau+\mu,\tau)=(-1)^{l\lambda }{\rm exp}(-\pi\sqrt{-1}l(2\lambda z+\lambda^2\tau))
{\rm Ell}_2(M,W,z,\tau),
\end{align}
\begin{align}
{\rm Ell}_3(M,W,z+\lambda \tau+\mu,\tau)={\rm exp}(-\pi\sqrt{-1}l(2\lambda z+\lambda^2\tau))
{\rm Ell}_3(M,W,z,\tau).
\end{align}
When $l$ is even, the elliptic genus ${\rm Ell}_1(M,W,\tau,z),~{\rm Ell}_2(M,W,\tau,z),~{\rm Ell}_3(M,W\tau,z)$ are $\Gamma_0(2)$, $\Gamma^0(2)$, $\Gamma _{\theta}$-Jacobi forms of weight $d$ and index $\frac{l}{2}$ respectively.
\end{thm}

\begin{proof}
By (2.12)-(2.24), we have
we get ${\rm Ell}_j(M,W,\tau,z)$ satisfies the following transformation laws:
\begin{align}
&{\rm Ell}_1(M,W,\tau+1,z)={\rm Ell}_1(M,W,\tau,z),\\\notag
&{\rm Ell}_2(M,W,\tau+1,z)={\rm Ell}_3(M,W,\tau,z),\\\notag
&{\rm Ell}_3(M,W,\tau+1,z)={\rm Ell}_2(M,W,\tau,z),\\\notag
&{\rm Ell}_1(M,W,\tau,z+\mu)=(-1)^{\mu}{\rm Ell}_1(M,W,V,\tau,z),\\\notag
&{\rm Ell}_2(M,W,\tau,z+\mu)={\rm Ell}_2(M,W,\tau,z),\\\notag
&{\rm Ell}_3(M,W,\tau,z+\mu)={\rm Ell}_3(M,W,\tau,z),\\\notag
&{\rm Ell}_1(M,W,z+\tau,\tau)={\rm exp}(-\pi\sqrt{-1}l(2z+\tau))
{\rm Ell}_1(M,W,z,\tau),\\\notag
&{\rm Ell}_2(M,W,z+\tau,\tau)=(-1)^l{\rm exp}(-\pi\sqrt{-1}l(2z+\tau))
{\rm Ell}_2(M,W,z,\tau),\\\notag
&{\rm Ell}_3(M,W,z+\tau,\tau)={\rm exp}(-\pi\sqrt{-1}l(2z+\tau))
{\rm Ell}_3(M,W,z,\tau),\\\notag
&{\rm Ell}_1(M,W,-\frac{1}{\tau},\frac{z}{\tau})=\tau^{d}{\rm exp}(\pi\sqrt{-1}l\frac{z^2}{\tau})
{\rm Ell}_2(M,W,\tau,z),\\\notag
&{\rm Ell}_2(M,W,V,-\frac{1}{\tau},\frac{z}{\tau})=\tau^{d}{\rm exp}(\pi\sqrt{-1}l\frac{z^2}{\tau})
{\rm Ell}_1(M,W,\tau,z),\\\notag
&{\rm Ell}_3(M,W,-\frac{1}{\tau},\frac{z}{\tau})=\tau^{d}{\rm exp}(\pi\sqrt{-1}l\frac{z^2}{\tau})
{\rm Ell}_3(M,W,\tau,z).
\end{align}
By (4.7), we can prove Theorem 4.3.
\end{proof}

Similar to Proposition 3.5, we have
\begin{prop}
If $c_1(W^{1,0})=0$ and the first Pontrjagin classes $p_1(M)=p_1(W)$, then the series $a_{n,1}(M,W,\tau)$, $a_{n,2}(M,W,\tau)$, $a_{n,3}(M,W,\tau)$ determined by
\begin{align}
{\rm exp}(-4\pi^2lG_2(\tau)z^2){\rm Ell}_a(M,W,\tau,z)=\sum_{n\geq 0}a_{n,\alpha}(M,W,\tau)\cdot z^n,~~1\leq a\leq 3
\end{align}
are modular forms of weight $d+n$ over $\Gamma_0(2)$, $\Gamma^0(2)$, $\Gamma _{\theta}$ respectively.
Let \begin{equation} a_{n,2}(M,W,\tau)=a_{n,2}^0+a_{n,2}^{\frac{1}{2}}q^{\frac{1}{2}}+a_{n,2}^1q+a_{n,2}^{\frac{3}{2}}q^{\frac{3}{2}}+\cdots+a_{n,2}^mq^{\frac{m}{2}}+\cdots.
\end{equation}
\begin{align}
&a_{0,2}^{0}=\int_M\widehat{A}(TM),\\\notag
&a_{0,2}^{1}=
\int_M\widehat{A}(TM){\rm ch}(\wedge^2{W^{1,0}}+W^{1,0}\otimes W^{1,0}+\wedge^2{W^{0,1}}-2lW\otimes C+\widetilde{T_CM}+l(2l+1)),\\\notag
\end{align}
\begin{align}
&a_{1,2}^{0}=a_{3,2}^{0}=0,~a_{2,2}^{0}=\frac{l}{6}\pi^2\int_M\widehat{A}(TM),~a_{4,2}^{0}=\frac{l^2}{12}\pi^4\int_M\widehat{A}(TM).
\end{align}
\begin{align}
&a_{0,2}^{\frac{1}{2}}=\int_M\widehat{A}(TM){\rm ch}(2l-W\otimes C),\\\notag
&a_{1,2}^{\frac{1}{2}}=-2\pi\sqrt{-1}
\int_M\widehat{A}(TM){\rm ch}(W^{1,0}-W^{0,1}),\\\notag
&a_{2,2}^{\frac{1}{2}}=(2-\frac{l}{6})\pi^2\int_M\widehat{A}(TM){\rm ch}(W\otimes C)+\frac{l^2}{3}\pi^2\int_M\widehat{A}(TM),
\end{align}
\begin{align}
&a_{1,2}^{1}=2\pi\sqrt{-1}
\int_M\widehat{A}(TM){\rm ch}(\wedge^2{W^{1,0}}-\wedge^2{W^{0,1}}-2lW^{1,0}+2lW^{0,1}),
\end{align}
\end{prop}
\begin{proof}
 If we set
\begin{equation}{\rm Ell}_2(M,W,V,\tau,z):={B_0}(Z)+B_{\frac{1}{2}}(z)q^{\frac{1}{2}}+B_1(z)q+O(q^{\frac{3}{2}}),
\end{equation}
then
\begin{align}
&B_0(z)=\int_M\widehat{A}(TM),~~{B_{\frac{1}{2}}}(z)=\int_M\widehat{A}(TM){\rm ch}(2l-yW^{1,0}-y^{-1}W^{0,1}),\\\notag
&{B_1(z)}=\int_M\widehat{A}(TM){\rm ch}(y^2\wedge^2W^{1,0}+W^{1,0}\otimes W^{0,1}+y^{-2}\wedge^2W^{0,1}\\\notag
&-2l(yW^{1,0}+y^{-1}W^{0,1})+l(2l+1)+\widetilde{T_CM}).
\end{align}
We know that
\begin{equation}
\sum_{n\geq 0}a_{n,2}^0z^n=\widehat{A}_0(z)B_0(z),~~\sum_{n\geq 0}a_{n,2}^{\frac{1}{2}}z^n=\widehat{A}_0(z)B_{\frac{1}{2}}(z),~~\sum_{n\geq 0}a_{n,2}^1z^n=\widehat{A}_0(z)B_1(z)+\widehat{A}_1(z)B_0(z),
\end{equation}
then we can get Proposition 4.4 by (3.38), (3.40), (4.14)-(4.16).
\end{proof}

\begin{prop}
If $c_1(W^{1,0})=0$ and the first Pontrjagin classes $p_1(M)=p_1(W)$, then\\
1)if $d$ is odd, then
\begin{align}
&a_{0,2}^{1}=
\int_M\widehat{A}(TM){\rm ch}(\wedge^2{W^{1,0}}+W^{1,0}\otimes W^{1,0}+\wedge^2{W^{0,1}}-2lW\otimes C+\widetilde{T_CM}+l(2l+1))=0.
\end{align}
2)if $d$ is even, then
\begin{align}
&a_{1,2}^{\frac{1}{2}}=-2\pi\sqrt{-1}
\int_M\widehat{A}(TM){\rm ch}(W^{1,0}-W^{0,1})=0,\\\notag
&a_{1,2}^{1}=2\pi\sqrt{-1}
\int_M\widehat{A}(TM){\rm ch}(\wedge^2{W^{1,0}}-\wedge^2{W^{0,1}}-2lW^{1,0}+2lW^{0,1})=0.
\end{align}
\end{prop}

\begin{thm}If $c_1(W^{1,0})=0$ and the first Pontrjagin classes $p_1(M)=p_1(W)$, then,\\
1)if $d=2$, then
\begin{align}
&\int_M\widehat{A}(TM){\rm ch}(2l-W\otimes C)=24\int_M\widehat{A}(TM)\\\notag
&\int_M\widehat{A}(TM){\rm ch}(\wedge^2{W^{1,0}}+W^{1,0}\otimes W^{1,0}+\wedge^2{W^{0,1}}-2lW\otimes C+\widetilde{T_CM}+l(2l+1))\\\notag
&=24\int_M\widehat{A}(TM).
\end{align}
That is $a^{\frac{1}{2}}_{0,2}=24a^{0}_{0,2}$ and $a^{1}_{0,2}=24a^{0}_{0,2}$ and $a^{\frac{1}{2}}_{0,2}$, $a^{1}_{0,2}$ are the integer multiple of $24$.\\
2)if $d=4$, then $240a^{0}_{0,2}+8a^{\frac{1}{2}}_{0,2}-a^{1}_{0,2}=0$.\\
3)if $d=6$, then $504a^{0}_{0,2}-32a^{\frac{1}{2}}_{0,2}+a^{1}_{0,2}=0$.\\
\end{thm}

\noindent{\bf Remark: }When $M$ is an almost complex spin manifold and $c_1(T^{1,0}M)=0$, then $\widehat{A}(TM)={\rm Todd}(T^{1,0}M)$ and above two-variable elliptic genera become the Todd genus for almost complex manifolds. \\

\section{Acknowledgements}

 The author was supported by Science and Technology Development Plan Project of Jilin Province, China: No.20260102245JC and NSFC No.11771070.
\vskip 1 true cm


\bigskip
\bigskip

 \indent{School of Mathematics and Statistics,
Northeast Normal University, Changchun Jilin, 130024, China }\\
\indent E-mail: {\it wangy581@nenu.edu.cn }\\

\end{document}